\documentclass{amsart}
\usepackage{amssymb}
\usepackage{graphicx}

\newtheorem{theorem}{Theorem}[section]
\newtheorem{proposition}[theorem]{Proposition}
\newtheorem{corollary}[theorem]{Corollary}
\newtheorem{algorithm}[theorem]{Algorithm}

\newtheorem{preremark}[theorem]{Remark}
\newtheorem{predefinition}[theorem]{Definition}
\newtheorem{preexample}[theorem]{Example}

\newenvironment{definition}{\begin{predefinition}\rm}{\end{predefinition}}
\newenvironment{example}{\begin{preexample}\rm}{\end{preexample}}

\def\LT{\mathop{\rm LT}\nolimits}
\def\LC{\mathop{\rm LC}\nolimits}

\def\NR{\mathop{\rm NR}\nolimits}

\def\Mat{\mathop{\rm Mat}\nolimits}

\def\apker{\mathop{\rm apker}\nolimits}
\def\eval{\mathop{\rm eval}\nolimits}
\def\Supp{\mathop{\rm Supp}\nolimits}
\def\IX{I_{\mathbb{X}}}
\def\indOF{{\rm ind}_{\mathcal{O}_F}}

\let\epsilon=\varepsilon
\let\rho=\varrho
\let\phi=\varphi

\def\apcocoa{\mbox{\rm
   A\kern-.1em p\kern-.07em C\kern-.13em o\kern-.07 em
   C\kern-.13em o\kern-.15em A}}

\begin{document}

\title{Subideal Border Bases}

\author{Martin Kreuzer}
\address{Fakult\"at f\"ur Informatik und Mathematik,
Universit\"at Passau, D-94030 Passau,
Germany} \email{martin.kreuzer@uni-passau.de}

\author{Henk Poulisse}
\address{Shell Int.\ Exploration and Production, Exploratory Research,
Kessler Park 1, NL-2288 GD Rijswijk, The Netherlands}
\email{hennie.poulisse@shell.com}

\date{\today}

\begin{abstract}
In modeling physical systems, it is sometimes useful to
construct border bases of 0-dimensional polynomial ideals
which are contained in the ideal generated by a given set of
polynomials. We define and construct such
subideal border bases, provide some basic properties
and generalize a suitable variant of the Buchberger-M\"oller algorithm
as well as the AVI-algorithm of~\cite{HKPP} to the subideal setting.
The subideal version of the AVI-algorithm is then applied
to an actual industrial problem.
\end{abstract}


\maketitle


\section*{Contents}

\begin{enumerate}
{
\item[1.] Introduction
\item[2.] Subideal Border Bases
\item[3.] The Subideal Border Division Algorithm
\item[4.] The Subideal Version of the BM-Algorithm
\item[5.] The Subideal Version of the AVI-Algorithm
\item[6.] An Industrial Application
\item[] References
}
\end{enumerate}


\section{Introduction}

In~\cite{HKPP} an algorithm was introduced which
computes an approximate border basis consisting of unitary
polynomials that vanish approximately at a given set
of points. It has been shown that this {\it AVI-algorithm}\/
is useful for modeling physical systems based on a set of
measured data points. More precisely, given a finite point set
$\mathbb{X}=\{p_1,\dots,p_s\} \subset [-1,1]^n$, the
AVI-algorithm computes an order ideal~$\mathcal{O}$ of terms
in~$\mathbb{T}^n$ and an $\mathcal{O}$-border prebasis
$G=\{g_1,\dots,g_\nu\}$ such that

\begin{enumerate}
\item the unitary polynomials $g_i/\| g_i\|$ vanish
$\epsilon$-approximately at~$\mathbb{X}$, where $\epsilon>0$
is a given threshold number, and

\item the normal remainders of the S-polynomials $S(g_i,g_j)$
for $g_i,g_j$ with neighboring border terms are
smaller than~$\epsilon$.
\end{enumerate}

Abstractly speaking, the last condition means that
the point in the moduli space corresponding to~$G$
is ``close'' to the border basis scheme (see~\cite{KR3} and~\cite{KPR}).
In practical applications, the AVI-algorithm turns out to be very
stable and useful. With a judicial choice of the threshold number~$\epsilon$,
it is able to discover simple polynomial relations which exist in the data
with high reliability.
For instance, it discovers simple physical laws inherent in measured data
without the need of imposing model equations.

However, in some situations physical information may be available
which is not contained in the data points~$\mathbb{X}$ or we may have
exact physical knowledge which is only approximately represented by
the data points.
An example for this phenomenon will be discussed in Section~6.
For instance, we may want to
{\it impose}\/ certain vanishing conditions on the model equations
we are constructing. Using Hilbert's Nullstellensatz this translates
to saying that what we are looking for is the intersection
of the vanishing ideal of~$\mathbb{X}$ with a given ideal
$J\subseteq \mathbb{R}[x_1,\dots,x_n]$ whose generators
represent the vanishing conditions we want to impose.

In order to be able to deal with this approximate situation,
it is first necessary to generalize the exact version of the
computation of vanishing ideals to the subideal setting.
Then this theory will serve as a guide and a motivation for the
approximate case. Therefore this paper begins in Section~2 with the
definition and basic properties of subideal border bases.

Given a 0-dimensional ideal $I$ in a polynomial ring
$P=K[x_1,\dots,x_n]$ over a field and a set of polynomials
$F=\{f_1,\dots,f_m\}$ generating an ideal $J=\langle F\rangle$,
a subideal border basis of~$I$ corresponds to a set of
polynomials $\mathcal{O}_F=\mathcal{O}_1\cdot f_1 \cup \cdots
\cup \mathcal{O}_m\cdot f_m$, where the $\mathcal{O}_i$ are
order ideals of terms, such that the residue classes of
the elements of~$\mathcal{O}_F$ form a $K$-basis of
$J/(I\cap J)\cong (I+J)/I$. Clearly, this generalizes the
case $F=\{1\}$, i.e.\ the ``usual'' border basis theory.
We show that subideal border bases always exist and
explain a method to construct them from a border basis of~$I$.
Moreover, we discuss some uniqueness properties of subideal
border bases.

The foundation of any further development of the theory of
subideal border bases is a generalization of the
Border Division Algorithm (see~\cite{KR2}, 6.4.11)
to the subideal case. This foundation is laid in Section~3 where
we also study higher $\mathcal{O}_F$-borders, the
$\mathcal{O}_F$-index, and show that a subideal
border basis of~$I$ generates $I\cap J$.

In Section~4, we generalize the Buchberger-M\"oller algorithm
(BM-algorithm) for computing vanishing ideals of point sets to the
subideal setting. More precisely, we generalize a version of
the BM-algorithm which proceeds blockwise degree-by-degree and
produces a border basis of the vanishing ideal.
Similarly, the subideal version of the BM-algorithm (cf.\
Algorithm~\ref{SBMalg}) computes an
$\mathcal{O}_\sigma(\IX)_F$-subideal border basis of~$\IX$, where
$\mathcal{O}_\sigma(\IX)$ is the complement of a leading term
ideal of the vanishing ideal~$\IX$ of~$\mathbb{X}$.

Next, in Section~5, we turn to the setting of Approximate Computational Algebra.
We work in the polynomial ring $\mathbb{R}[x_1,\dots,x_n]$ over the reals
and assume that $\mathbb{X}\subset [-1,1]^n$ is a finite set of (measured, imprecise)
points. We define approximate $\mathcal{O}_F$-subideal border bases
and generalize the AVI-algorithm from~\cite{HKPP}, Thm.~3.3 to the subideal
case.

Let us point out that the subideal version of the AVI algorithm
contains a substantial difference to the traditional way of
computing approximate vanishing ideals, e.g.\ as in~\cite{AFT}.
Namely, the AVI algorithm produces a set of polynomials which
vanish approximately at the given data points, but we do not
demand that there exists a ``nearby'' set of points at which
these polynomials vanish exactly. The latter requirement
has turned out to be too restrictive for real-world applications,
for instance the one we explain in the last section. There we provide an example for
the application of these techniques to the problem of production allocation
in the oil industry.

Unless explicitly stated otherwise, we use the notation and
definitions of~\cite{KR1} and~\cite{KR2}. We shall assume that
the reader has some familiarity with the theory of exact and
approximate border bases (see for instance \cite{HKPP},
\cite{KK}, \cite{KKR}, \cite{KPR},
Section~6.4 of~\cite{KR2}, and~\cite{St}).

\bigskip

\section{Subideal Border Bases}

Here we are interested in a ``relative'' version
of the notion of border bases in the following sense.
Let $K$ be a field, let $P=K[x_1,\dots,x_n]$ be a
polynomial ring, let $\mathbb{T}^n$ be its monoid of terms,
let $\mathcal{O}$ be an order ideal in~$\mathbb{T}^n$, and let
$I\subset P$ be a 0-dimensional ideal.

Suppose we are given a further polynomial ideal
$J=\langle f_1,\dots,f_m\rangle$ of~$P$,
where $F=\{f_1,\dots,f_m\}\subset P\setminus \{0\}$.
Our goal is to describe and compute the intersection
ideal $I\cap J$ as a subideal of~$J$.
By Noether's isomorphism theorem, we have
$J/(I\cap J) \cong (I+J)/I \subset P/I$.
Therefore $J$ has a finite $K$-vector space basis
modulo $I\cap J$. Now we are looking for the
following special kind of vector space basis.

\begin{definition} Let $\mathcal{O}$ be an order
ideal of terms in~$\mathbb{T}^n$ whose residue classes form
a $K$-vector space basis of~$P/I$.

\begin{enumerate}
\item For $i=1,\dots,m$, let $\mathcal{O}_i
\subseteq \mathcal{O}$ be an order ideal.
Then the set $\mathcal{O}_F=\mathcal{O}_1\cdot f_1
\cup \cdots \cup \mathcal{O}_m\cdot f_m$
is called an {\it $F$-order ideal}.
Its elements, i.e.\ products of the form $tf_i$
with $t\in\mathcal{O}_i$ will be called {\it $F$-terms}.

\item If $\mathcal{O}_F=\mathcal{O}_1\cdot f_1
\cup \cdots \cup \mathcal{O}_m\cdot f_m$
is an $F$-order ideal whose residue classes
form a $K$-vector space basis of $J/(I\cap J)$,
we say that the ideal~$I$ has an {\it $\mathcal{O}_F$-subideal
border basis}.
\end{enumerate}
\end{definition}

Notice that an $F$-term may be viewed as a generalization
of the usual notion of term by using $F=\{1\}$.
Similarly, $F$-order ideals generalize the usual
order ideals.
It is natural to ask whether every ideal~$I$
supporting an $\mathcal{O}$-border bases has
an $\mathcal{O}_F$-subideal border basis
for some $F$-order ideal $\mathcal{O}_F$.
The next proposition answers this positively.

\begin{proposition}
Let $I\subset P$ be a 0-dimensional ideal,
and let $J=\langle f_1,\dots,f_m\rangle\subset P$
be any ideal.

\begin{enumerate}
\item Given an order ideal $\mathcal{O}\subset
\mathbb{T}^n$ whose residue classes generate the
$K$-vector space $P/I$, there exists an order ideal
$\widetilde{\mathcal{O}} \subseteq \mathcal{O}$
whose residue classes from a $K$-vector space
basis of~$P/I$.

\item Let $U\subset P^m$ be a $P$-submodule, and let
$\mathcal{O}_1,\dots,\mathcal{O}_m$ be order ideals
in~$\mathbb{T}^n$ such that the residue classes of
$\mathcal{O}_1 e_1 \cup \cdots \cup \mathcal{O}_m e_m$
generate the $K$-vector space $P^m/U$. Then there exist
order ideals $\widetilde{\mathcal{O}_i}\subseteq
\mathcal{O}_i$ such that the residue classes of
$\widetilde{\mathcal{O}}_1 e_1 \cup \cdots \cup
\widetilde{\mathcal{O}}_m e_m$
form a $K$-vector space basis of $P^m/U$.

\item If $\mathcal{O}$ is an order ideal whose residue classes
from a $K$-vector space basis of~$P/I$ then there exist order ideals
$\mathcal{O}_i\subseteq \mathcal{O}$ such that the residue
classes of $\mathcal{O}_F=\mathcal{O}_1 f_1 \cup \cdots\cup
\mathcal{O}_m f_m$ are a $K$-vector space basis of $J/(I\cap J)$.
In other words, the ideal~$I$ has an $\mathcal{O}_F$-subideal
border basis.

\end{enumerate}
\end{proposition}

\begin{proof}
First we show (1).  We construct the order ideal~$\widetilde{\mathcal{O}}$
inductively. To this end, we choose a degree compatible term
ordering~$\sigma$ and order~$\mathcal{O}=\{t_1,\dots,t_\mu\}$
such that $t_1 <_\sigma \cdots <_\sigma t_\mu$. In particular,
we have $t_1=1$. Since $I\subset P$, we can start by putting~$t_1$
into~$\widetilde{\mathcal{O}}$ and removing it from~$\mathcal{O}$.

For the induction step, we consider the $\sigma$-smallest term~$t_i$
which is still in~$\mathcal{O}$. If the residue class of~$t_i$ in~$P/I$
is $K$-linearly dependent on the residue classes of the elements
in~$\widetilde{\mathcal{O}}$, we cancel~$t_i$ and all of its multiples
in~$\mathcal{O}$. Each of these terms can be rewritten modulo~$I$
as a linear combination of smaller terms w.r.t.~$\sigma$.
If the residue class of~$t_i$ in~$P/I$ is $K$-linearly independent of
the residue classes of the elements in~$\widetilde{\mathcal{O}}$,
we append~$t_i$ to~$\widetilde{\mathcal{O}}$ and remove it
from~$\mathcal{O}$. In this way, the residue classes of the elements
of~$\widetilde{\mathcal{O}}$ are always $K$-linearly independent
in~$P/I$, and every element of~$\mathcal{O}$ can be rewritten modulo~$I$
as a $K$-linear combination of the elements of the final
set~$\widetilde{\mathcal{O}}$.

Now we show~(2). Let $\overline{U}$ be the idealization of~$U$
in $\overline{P}=K[x_1,\dots,x_n,e_1,\dots,e_m]$ (see~\cite{KR2},
Section 4.7.B). The residue classes of the elements of
the order ideal $\overline{\mathcal{O}}=
\mathcal{O}_1e_1 \cup \cdots \cup \mathcal{O}_m e_m$
generate the $K$-vector space $\overline{P}/\overline{U}$.
Now it suffices to apply~(1) and to note that every
order subideal of~$\overline{\mathcal{O}}$ has the indicated form.

Finally, we prove~(3). Consider the $P$-linear
map $\phi: P^m \longrightarrow P/I$ defined by $e_i \mapsto f_i+I$.
Its image is the ideal $(I+J)/I$. Let $U=\ker(\phi)$.
Then~$\phi$ induces an isomorphism $\bar{\phi}: P^m/U \cong (I+J)/I$.
To get a system of generators of~$P^m/U$, it suffices to
find a system of generators of the ideal
$(I+J)/I = \langle f_1+I,\dots,f_m+I \rangle$.
As a vector space, this ideal is generated by
$\mathcal{O}\cdot (f_1+I) \cup\cdots\cup \mathcal{O}\cdot
(f_m+I)$. The preimages of these generators are
the elements of $\mathcal{O}e_1 \cup \cdots \cup \mathcal{O} e_m$.
Now an application of~(2) finishes the proof.
\end{proof}

Based on this proposition, we can construct an $F$-order ideal
such that a given ideal~$I$ has an $\mathcal{O}_F$-subideal
border basis. The following example illustrates the method.

\begin{example}\label{OFexample}
Let $P=\mathbb{Q}[x,y]$, let $I=\langle x^2-x,\, y^2-y\rangle$,
let $\mathcal{O}=\{1,\,x,\,y,\,xy\}$, and let $J=\langle x+y\rangle$.
Then~$I$ has an $\mathcal{O}$-border basis and therefore also
an $\mathcal{O}_F$-subideal border basis w.r.t.\ $F=\{f\}$
for $f=x+y$.

To construct a suitable $F$-order ideal, we start with $\mathcal{O}_F=
\{1\cdot f\}$. Then we put $x\cdot f$ and $y\cdot f$ into~$\mathcal{O}_F$,
since we have $x\cdot f \equiv xy+x$ and $y\cdot f\equiv xy+y$ modulo~$I$,
and since $\{x+y,\,xy+x,\,xy+y\}$ is $\mathbb{Q}$-linearly independent
in~$P/I$. Next $xy\cdot f \equiv 2xy \equiv x\cdot f + y\cdot f -1\cdot f$
implies that we are done. The result is that
$\mathcal{O}_F=\{f,\,xf,\,yf\}$ is an $F$-order ideal for which~$I$
has an $\mathcal{O}_F$-subideal border basis.
\end{example}

At this point it is time to explain the choice of the
term ``subideal border basis'' in the above definition.

\begin{definition}
Let $F=\{f_1,\dots,f_m\} \subset P\setminus \{0\}$, and let
$\mathcal{O}_F=\mathcal{O}_1 f_1 \cup \cdots \cup
\mathcal{O}_m f_m$ be an $F$-order ideal.
We write $\mathcal{O}_F=\{t_1 f_{\alpha_1},\dots,
t_\mu f_{\alpha_\mu}\}$ with $\alpha_i\in\{1,\dots,m\}$
and $t_i \in \mathcal{O}_{\alpha_i}$.

\begin{enumerate}
\item The set of polynomials $\partial\mathcal{O}_F=
(x_1\mathcal{O}_F \cup \cdots\cup x_n \mathcal{O}_F)
\setminus \mathcal{O}_F$
is called the {\it border} of~$\mathcal{O}_F$.

\item Let $\partial \mathcal{O}_F= \{b_1 f_{\beta_1},\dots,b_\nu f_{\beta_\nu}\}$.
A set of polynomials $G=\{g_1,\dots,g_\nu\}$ is called
an {\it $\mathcal{O}_F$-subideal border prebasis} if
$g_j= b_j f_{\beta_j} - \sum_{i=1}^\mu c_{ij} t_i f_{\alpha_i}$
with $c_{1j},\dots,c_{\mu j}\in K$ for $j=1,\dots,\nu$.

\item An $\mathcal{O}_F$-subideal border prebasis~$G$ is called an
{\it $\mathcal{O}_F$-subideal border basis} of an ideal~$I$
if~$G$ is contained in~$I$ and the residue classes of the elements
of~$\mathcal{O}_F$ form a $K$-vector space basis of $J/(I\cap J)$.

\end{enumerate}
\end{definition}

In this terminology, the last part of the preceding
proposition can be rephrased as follows.

\begin{corollary}
Let $\mathcal{O}$ be an order ideal in~$\mathbb{T}^n$,
and let $I\subset P$ be a 0-dimensional ideal which
has an $\mathcal{O}$-border basis. Then~$I$ has an
$\mathcal{O}_F$-subideal border basis for every ideal
$J=\langle f_1,\dots,f_m \rangle$ and $F=\{f_1,\dots,f_m\}
\subset P\setminus \{0\}$.
\end{corollary}

In the setting of Example~\ref{OFexample}, the
$\mathcal{O}_F$-subideal border basis of~$I$
can be constructed as follows.

\begin{example}
The border of the $F$-order ideal
$\mathcal{O}_F=\{f,xf,yf\}$ is
$\partial\mathcal{O}_F=\{x^2f,xyf,y^2f\}$.
We compute modulo~$I$ and find
$x^2f \equiv xf$, $xyf\equiv xf+yf-f$, and
$y^2f\equiv yf$. Therefore the set
$G=\{x^2f-xf,\, xyf-xf-yf+f,\, y^2f-yf\}$
is an $\mathcal{O}_F$-subideal border basis
of~$I$.
\end{example}

If an ideal has an $\mathcal{O}_F$-subideal border basis,
the elements of this basis are uniquely determined.
This follows exactly as in the
case $J=\langle 1\rangle$, i.e.\ the case of the usual border
bases (see~\cite{KR2}, 6.4.17 and 6.4.18). Notice, however,
that a set of polynomials may be an $F$-order ideal
in several different ways. This is illustrated by the
following example.

\begin{example}
Let $P=\mathbb{Q}[x,y]$, let
$I=\langle x^2-x,\,y^2-y\rangle$, and let $J=\langle x,\,y\rangle$.
Clearly, the ideal~$I$ has an $\mathcal{O}$-border basis
for $\mathcal{O}=\{1,\,x,\,y,\,xy\}$, namely the set
$G=\{x^2-x,\, x^2y-xy,\, xy^2-xy,\, y^2-y\}$.
Hence the ideal~$I$ also has an $\mathcal{O}_F$-subideal
border basis for $F=\{x,\,y\}$. Here we can use both
$\mathcal{O}_F=\{1,y\}\cdot x \cup \{1\}\cdot y$
and $\mathcal{O}_F=\{1\}\cdot x \cup \{1,x\}\cdot y$.

This example shows also another phenomenon: an $F$-term
can simultaneously be contained in~$\mathcal{O}_F$ and in
$\partial\mathcal{O}_F$. For instance, if we use
$\mathcal{O}_F=\{1,y\}\cdot x \cup \{1\}\cdot y$, the term
$xy$ is both contained in $\{1,y\}\cdot x$ and in the
border of $\{1\}\cdot y$. The resulting subideal border
basis will contain the polynomial $xy-xy=0$.
\end{example}

Finally, we give an example where a term is in~$\partial\mathcal{O}_F$
in two different ways, so that a subideal border basis polynomial
is repeated.

\begin{example}
Let $I=\langle x^2-x,\,y^2-y,\,xy\rangle \subseteq \mathbb{Q}[x,y]$,
and let $J=\langle x,\, y\rangle \subset \mathbb{Q}[x,y]$.
Then the subideal border basis of~$I$ with respect to
$\mathcal{O}_F=\{1\}\cdot x \cup \{1\}\cdot y$ is
$G=\{x^2-x,\,xy,xy,y^2-y\}$ where $xy$ appears both in
$\partial\{1\}\cdot x$ and in $\partial\{1\}\cdot y$.
\end{example}

\bigskip

\section{The Subideal Border Division Algorithm}

A central result in the construction of any Gr\"obner-basis-like
theory is a suitable version of the division algorithm (for the classical
case, see for instance~\cite{KR1}, Thm.~1.6.4 and specifically for
border bases, see~\cite{KR2}, Prop.~6.4.11). Before we can present
a subideal border basis version, we need a few additional definitions.

\begin{definition}\label{defhigherborder}
Let $F=\{f_1,\dots,f_m\} \subset P\setminus \{0\}$, and let
$\mathcal{O}_F$ be an $F$-order ideal.

\begin{enumerate}
\item The {\it first border closure} of~$\mathcal{O}_F$ is
$\overline{\partial \mathcal{O}_F}= \mathcal{O}_F\cup
\partial\mathcal{O}_F$.

\item For every $k\ge 1$, we inductively define the
{\it $(k+1)^{\it st}$ border} of~$\mathcal{O}_F$
by $\partial^{k+1}\mathcal{O}_F =
\partial (\overline{\partial^k \mathcal{O}_F})$ and the
{\it $(k+1)^{\it st}$ border closure} of~$\mathcal{O}_F$ by
$\overline{\partial^{k+1}\mathcal{O}_F} =
\overline{\partial^k\mathcal{O}_F} \cup \partial^{k+1}\mathcal{O}_F$.

\item Finally, we let $\partial^0 \mathcal{O}_F=
\overline{\partial^0\mathcal{O}_F} = \mathcal{O}_F$.
\end{enumerate}
\end{definition}

Using these higher borders,
the set $\mathbb{T}^n\, f_1 \cup \cdots\cup \mathbb{T}^n\, f_m$
is partitioned as follows.

\begin{proposition}\label{BorderProps}
Let $F=\{f_1,\dots,f_m\} \subset P\setminus \{0\}$, and let
$\mathcal{O}_F$ be an $F$-order ideal.

\begin{enumerate}
\item For every $k\ge 0$, we have a disjoint union
$\overline{\partial^k\mathcal{O}_F}=\bigcup_{i=0}^k\partial^i\mathcal{O}_F$.

\item For every $k\ge 0$, we have $\overline{\partial^k\mathcal{O}_F}
=\mathbb{T}^n_{\le k}\cdot \mathcal{O}_F$.

\item For every $k\ge 1$, we have
$\partial^k\mathcal{O}_F=\mathbb{T}_k^n\cdot \mathcal{O}_F
\setminus \mathbb{T}_{<k}^n\cdot \mathcal{O}_F$.

\item We have $\bigcup_{i=0}^m\mathbb{T}^n\cdot f_i
=\bigcup_{j=0}^\infty \partial^j\mathcal{O}_F$,
where the right-hand side is a disjoint union.

\item Any $F$-term $tf_i\in\mathbb{T}^n\cdot f_i\setminus \mathcal{O}_F$
is divisible by an $F$-term in~$\partial\mathcal{O}_F$.

\end{enumerate}
\end{proposition}

\begin{proof} First we show (1) by induction on~$k$. For $k=0$,
the claim follows from the definition. For $k=1$, we have
$\overline{\partial^1 \mathcal{O}_F}= \partial^0\mathcal{O}_F\cup
\partial^1\mathcal{O}_F$ by Definition~\ref{defhigherborder}.a.
Inductively, it follows that
$\overline{\partial^{k+1}\mathcal{O}_F}=
\overline{\partial^k\mathcal{O}_F}\cup \partial^{k+1}\mathcal{O}_F
= \bigcup_{i=0}^{k+1}\partial^i \mathcal{O}_F$.
This is a disjoint union, since
$\partial^{k+1}\mathcal{O}_F \cap \overline{\partial^k\mathcal{O}_F}
=\emptyset$ in each step.

Next we prove claim~(2). Again we proceed by induction on~$k$,
the case $k=0$ being obviously true. Inductively, we have
$\overline{\partial^{k+1}\mathcal{O}_F}=
\overline{\partial^k\mathcal{O}_F}\cup \partial^{k+1}\mathcal{O}_F
=\mathbb{T}^n_{\le k} \cdot \mathcal{O}_F \cup
\mathbb{T}^n_1\cdot (\mathbb{T}^n_{\le k} \cdot \mathcal{O}_F)
= \mathbb{T}^n_{\le k+1}\cdot \mathcal{O}_F$.

Claim~(3) is a consequence of~(2) and the equality
$\partial^k\mathcal{O}_F=
\overline{\partial^k\mathcal{O}_F}\setminus
\overline{\partial^{k-1}\mathcal{O}_F}$.
The fourth claim follows from the observation that, by~(2),
every $F$-term is in $\overline{\partial^k\mathcal{O}_F}$
for some $k\ge 0$.

Finally, claim~(5) holds because (4) implies that
$t f_i\in\partial^k \mathcal{O}_F$ for some $k\ge 1$,
and by~(3) this is equivalent to the existence
of a factorization $t=t't''$ where $\deg(t')=k-1$ and
$t''f_i\in \partial\mathcal{O}_F$.
\end{proof}

In view of this result, the following definition
appears natural.

\begin{definition}
Let $F=\{f_1,\dots,f_m\} \subset P\setminus \{0\}$, and let
$\mathcal{O}_F$ be an $F$-order ideal.

\begin{enumerate}
\item For an $F$-term $tf_i\in\mathcal{O}_F$,
we define $\indOF(tf_i)=
\min \{k\ge 0\mid tf_i\in\overline{\partial^k\mathcal{O}_F} \}$
and call it the {\it $\mathcal{O}_F$-index}\/
of~$tf_i$.

\item Given a non-zero polynomial $f\in J$, we write
$f=p_1 f_1 +\cdots +p_m f_m$ with $p_i\in P$
and we let $\mathcal{P}=(p_1f_1,\dots,p_mf_m)$.
Then the number
$$
\indOF(\mathcal{P})= \max\{ \indOF(t f_i) \mid
i\in \{1,\dots,m\},\; t\in \Supp(p_i)\}
$$
is called the {\it $\mathcal{O}_F$-index}\/ of the
{\it representation}~$\mathcal{P}$ of~$f$.
\end{enumerate}
\end{definition}

In other words, the $\mathcal{O}_F$-index of~$tf_i$ is the unique
number $k\ge 0$ such that $tf_i\in \partial^k\mathcal{O}_F$.
Note that the $\mathcal{O}_F$-index of a polynomial
$f\in J$ depends on the representation of~$f$ in terms
of the generators of~$J$. It is not clear how to find
a representation~$\mathcal{P}$ which yields the smallest
$\indOF(\mathcal{P})$. Using the Subideal Border Division
Algorithm, we shall address this point below.

The following proposition collects some basic properties
of the $\mathcal{O}_F$-index.

\begin{proposition}\label{IndexProps}
Let $F=\{f_1,\dots,f_m\} \subset P\setminus \{0\}$, and let
$\mathcal{O}_F$ be an $F$-order ideal.

\begin{enumerate}
\item For an $F$-term $tf_i\in\mathbb{T}^n\cdot f_i$,
the number $k=\indOF(t f_i)$ is
the smallest natural number such that there exists a
factorization $t=t' t''$ with a term $t'\in\mathbb{T}^n$
of degree~$k$ and with $t''f_i\in\mathcal{O}_F$.

\item Given $t\in\mathbb{T}^n$ and an $F$-term
$t'f_i\in \mathbb{T}^n\cdot f_i$, we have
$$
\indOF(t\,t'f_i) \le \deg(t)+\indOF(t'f_i).
$$

\item For $f,g\in J\setminus \{0\}$ such that $f+g\ne 0$, we
write $f=p_1f_1+\cdots + p_mf_m$ and $g=q_1f_1+\cdots+q_m f_m$
with $p_i,q_j\in P$, and we let $\mathcal{P}=(p_1f_1,\dots,p_mf_m)$
and $\mathcal{Q}=(q_1f_1,\dots,q_mf_m)$. Then we have
$$
\indOF(\mathcal{P}+\mathcal{Q})\le \max\{\indOF(\mathcal{P}),\,
\indOF(\mathcal{Q}) \}.
$$

\item Given $f\in J\setminus \{0\}$, we write
$f=p_1f_1+\cdots+p_mf_m$ with $p_i\in P$ and let
$\mathcal{P}=(p_1f_1,\dots,p_mf_m)$. For every $g\in P\setminus \{0\}$,
we then have
$$
\indOF(g\mathcal{P})\le \deg(g)+\indOF(\mathcal{P}).
$$
\end{enumerate}
\end{proposition}

\begin{proof}
The first claim follows from Prop.~\ref{BorderProps}.
The second claim follows from the first.
The third claim is a consequence of the fact that every $F$-term
appearing in $\mathcal{P}+\mathcal{Q}$ appears in~$\mathcal{P}$
or~$\mathcal{Q}$. The last claim follows from~(2) and
the observation that~$g\mathcal{P}$ is a $K$-linear combination of
tuples $t\mathcal{P}$ with $t\in\Supp(g)$.
\end{proof}

Now we have collected enough material to formulate and
prove the subideal version of the Border Division Algorithm.

\begin{algorithm}{\bf (The Subideal Border Division Algorithm)}\\
\label{sbdivalg}%
Let $F=\{f_1,\dots,f_m\}\subset P\setminus \{0\}$, let
$\mathcal{O}_F=\{t_1 f_{\alpha_1},\dots,t_\mu f_{\alpha_\mu}\}$
be an $F$-order ideal where $\alpha_i\in\{1,\dots,m\}$ and
$t_i\in\mathcal{O}_{\alpha_i}$, let $\partial\mathcal{O}_F=
\{b_1 f_{\beta_1},\dots,b_\nu f_{\beta_\nu}\}$ be its border, and let
$\{g_1,\dots,g_\nu\}$ be an $\mathcal{O}_F$-subideal border prebasis,
where $g_j= b_j f_{\beta_j} - \sum_{i=1}^\mu c_{ij} t_i f_{\alpha_i}$
with $c_{1j},\dots,c_{\mu j}\in K$ for $j=1,\dots,\nu$.
Given a polynomial $f\in J$, we write $f=p_1 f_1+\cdots+p_m f_m$
and consider the following instructions.

\begin{enumerate}
\item[{\bf D1}] Let $h_1=\cdots=h_\nu=0$, $c_1=\cdots =c_\mu=0$,
and $\mathcal{Q}=(q_1f_1,\dots,q_mf_m)$ with $q_i=p_i$
for $i=1,\dots,m$.

\item[{\bf D2}] If $\mathcal{Q}=(0,\dots,0)$
then return $(h_1,\dots,h_\nu,c_1,\dots,c_\mu)$ and stop.

\item[{\bf D3}] If $\indOF(\mathcal{Q})=0$ then find $c_1,\dots,c_\mu\in K$
such that $q_1f_1+\cdots+q_mf_m=c_1 t_1 f_{\alpha_1} + \cdots
+c_\mu t_\mu f_{\alpha_\mu}$.
Return $(h_1,\dots,h_\nu,c_1,\dots,c_\mu)$ and stop.

\item[{\bf D4}] If $\indOF(\mathcal{Q})>0$ then determine the
smallest index $i\in\{1,\dots,m\}$ such that there exists
a term $t\in\Supp(q_i)$ with $\indOF(tf_i)=\indOF(\mathcal{Q})$.
Choose such a term~$t$. Let $a\in K$ be the coefficient of~$t$ in~$q_i$.
Next, determine the smallest index
$j\in\{1,\dots,\nu\}$ such that~$t$ factors as
$t=t'\,t''$ with a term~$t'$ of degree $\indOF(tf_i)-1$
and with $t''f_i=b_j f_{\beta_j}\in\partial\mathcal{O}_F$.
Subtract the tuple corresponding to the representation
$$
a\, t'\, g_j= a\, t'\, b_j f_{\beta_j} - \sum_{i=1}^\mu
c_{ij}\, a\, t'\, t_i f_{\alpha_i}
$$
from~$\mathcal{Q}$, add~$a t'$ to~$h_j$,
and continue with step~{\bf D2}.

\end{enumerate}
\noindent This is an algorithm which returns a tuple
$(h_1,\dots,h_\nu,c_1,\dots,c_\mu)\in P^\nu \times K^\mu$ such that
$$
f= h_1 g_1 + \cdots + h_\nu g_\nu + c_1 t_1 f_{\alpha_1}+\cdots
+ c_\mu t_\mu f_{\alpha_\mu}
$$
and $\deg(h_i)\le \indOF(\mathcal{P})-1$ for $\mathcal{P}=
(p_1f_1,\dots,p_mf_m)$ and for all~$i\in\{1,\dots,\nu\}$
with $h_i\ne 0$. This representation does not depend on the
choice of the term~$t$ in step~{\bf D4}.
\end{algorithm}

\begin{proof}
First we show that all steps can be executed. In step~{\bf D3},
the condition $\indOF(\mathcal{Q})=0$ implies that
all $F$-terms $tf_i$ with $t\in\Supp(q_i)$ are contained
in~$\mathcal{O}_F$. In step~{\bf D4}, the definition of
$\indOF(\mathcal{Q})$ implies that a term~$t$ of the desired
kind exists. By Proposition~\ref{IndexProps}.1, this term~$t$
has a factorization $t=t't''$ with the desired properties.

Next we prove termination by showing that step~{\bf D4}
is performed only finitely many times. Let us investigate the
subtraction of the representation of $at'g_j$ from~$\mathcal{Q}$.
By the choice of~$t'$, the $\mathcal{O}_F$-index of $t' b_j f_{\beta_j}$
is $\deg(t')$ more than the $\mathcal{O}_F$-index
of $b_j f_{\beta_j}$. By Prop.~\ref{IndexProps}.b, this is
the maximal increase, and the $\mathcal{O}_F$-index of the other
$F$-terms in the representation of~$at'g_j$ is smaller
than $\indOF(\mathcal{Q})$. Thus the number of $F$-terms
in~$\mathcal{Q}$ of maximal $\mathcal{O}_F$-index
decreases by the subtraction, and after finitely many steps
the algorithm reaches step~{\bf D2} or~{\bf D3} and stops.

Finally, we prove correctness. To do so, we show that the equality
$$
f=q_1f_1+\cdots+q_mf_m + h_1 g_1 + \cdots + h_\nu g_\nu
+ c_1 t_1 f_{\alpha_1} +\cdots + c_\mu t_\mu f_{\alpha_\mu}
$$
is an invariant of the algorithm. It is satisfied at the end of step~{\bf D1}.
The constants $c_1,\dots,c_\mu$ are only changed in step~{\bf D3}.
In this case the contribution $q_1f_1 + \cdots + q_mf_m$
to the above equality is replaced by the equal contribution
$c_1 t_1 f_{\alpha_1} +\cdots + c_\mu t_\mu f_{\alpha_\mu}$.
The tuple~$\mathcal{Q}$ is only changed in step~{\bf D4}.
There the subtraction of the representation of~$at'g_j$
from~$\mathcal{Q}$ and the corresponding change in
$q_1f_1+\cdots+q_m f_m$ are compensated by the
addition of $at'$ to~$h_j$ and the corresponding change
in $h_1g_1+\cdots+h_\nu g_\nu$.
When the algorithm stops, we have $q_1=\cdots=q_m=0$.
This proves the claimed representation of~$f$.
Moreover, only terms of degree
$\deg(t')\le \indOF(\mathcal{Q})-1\le \indOF(\mathcal{P})-1$ are
added to~$h_j$.

The additional claim that the result of the algorithm does not
depend on the choice of~$t$ in step~{\bf D4}
follows from the observation that~$tf_i$ is replaced by $F$-terms
of strictly smaller $\mathcal{O}_F$-index. Thus the different
executions of step~{\bf D4} corresponding to the reduction of
several $F$-terms of maximal $\mathcal{O}_F$-index in~$\mathcal{Q}$
do not interfere with one another, and the final result -- after
all those $F$-terms have been rewritten -- is independent of the
order in which they are taken care of.
\end{proof}

Notice that in step~{\bf D4} the algorithm uses a term~$t$
which is not uniquely determined.
Also there may be several factorizations of~$t$.
We choose the indices~$i$ and~$j$ minimally to determine
this step of the algorithm uniquely, but this particular choice
is not forced upon us. Moreover, it is clear that the result of
the division depends on the numbering of the elements
of~$\partial\mathcal{O}_F$.

As indicated above, the Subideal Border Division Algorithm has
important implications. The following corollaries comprise a few
of them.

\begin{corollary}{\bf (Subideal Border Bases and Special Generation)}%
\label{specialgen}\\
In the setting of the algorithm, let $I=\langle G\rangle$.
Then the set~$G$ is an $\mathcal{O}_F$-subideal
border basis of~$I$ if and only if one of the following equivalent
conditions is satisfied.

\begin{enumerate}
\item[$(A_1)$] For every non-zero polynomial $f\in I\cap J$
with a representation $f=p_1 f_1+\cdots+p_m f_m$ and
$\mathcal{P}=(p_1 f_1,\dots,p_m f_m)$,
there exist polynomials $h_1,\dots,h_\nu\in P$ such that
$f=h_1g_1 +\cdots + h_\nu g_\nu$ and $\deg(h_i)\le
\indOF(\mathcal{P})-1$ whenever $h_i g_i\ne 0$.

\item[$(A_2)$] For every non-zero polynomial $f\in I\cap J$
with a representation $f=p_1 f_1+\cdots+p_m f_m$ and
$\mathcal{P}=(p_1 f_1,\dots,p_m f_m)$,
there exist $h_1,\dots,h_\nu\in P$ such that
$f=h_1g_1 +\cdots + h_\nu g_\nu$ and
$\max\{\deg(h_i)\mid i\in\{1,\dots,\nu\},\,h_i g_i\ne 0\} =
\indOF(\mathcal{P})-1$.
\end{enumerate}
\end{corollary}

\begin{proof}
First we show that $(A_1)$ holds if~$G$ is an
$\mathcal{O}_F$-border basis.
The Subideal Border Division Algorithm computes a representation
$f=h_1g_1+\cdots +h_\nu g_\nu + c_1 t_1 f_{\alpha_1}+\cdots
+c_\mu t_\mu f_{\alpha_\mu}$
with $h_1,\dots,h_\nu \in P$ and $c_1,\dots,c_\mu\in K$
such that $\deg(h_i)\le \indOF(\mathcal{P})-1$ for $i=1,\dots,\nu$.
Then $c_1 t_1 f_{\alpha_1} +\cdots + c_\mu t_\mu f_{\alpha_\mu}
\equiv 0$ modulo~$I$, and the
hypothesis implies $c_1=\cdots =c_\mu=0$.

Next we prove that $(A_1)$ implies $(A_2)$.
If $\deg(h_i)< \indOF(\mathcal{P})-1$,
then Prop.~\ref{IndexProps}.2 shows that
the $\mathcal{O}_F$-index of every representation
of~$h_ig_i$ is at most $\deg(h_i)+1$ and hence
smaller than $\indOF(\mathcal{P})$.
By Prop.~\ref{IndexProps}.4, there has to be at least one
number $i\in\{1,\dots,\nu\}$ such that
$\deg(h_i)=\indOF(\mathcal{P})-1$.

Finally, we assume $(A_2)$ and show the subideal border basis
property. Let $c_1,\dots,c_\mu\in K$
satisfy $c_1 t_1 f_{\alpha_1}+\cdots +c_\mu t_\mu f_{\alpha_\mu}\in I\cap J$.
Then either $f=c_1 t_1 f_{\alpha_1}+\cdots+c_\mu t_\mu f_{\alpha_\mu}$
equals the zero polynomial or not.
In the latter case we apply $(A_2)$ and obtain
a representation~$f=h_1g_1+\cdots+h_\nu g_\nu$ with
$h_1,\dots,h_\nu\in P$. Since $f\ne 0$,
we have $\max\{\deg(h_i)\mid i\in\{1,\dots,\nu\},\,h_i g_i\ne 0\}
\ge 0$. But $\indOF(\mathcal{P})-1=-1$ is in contradiction to
the second part of~$(A_2)$.
Hence we must have $f=0$. Thus $I\cap J \cap\langle \mathcal{O}_F\rangle_K=0$,
i.e.\ the set~$G$ is an $\mathcal{O}_F$-subideal border basis of~$I$.
\end{proof}

\begin{definition}
In the setting of the algorithm, let $\mathcal{G}=
(g_1,\dots,g_\nu)$. Then the polynomial
$$
\NR_{\mathcal{O}_F,\mathcal{G}}(\mathcal{P})=
c_1 t_1 f_{\alpha_1} + \cdots + c_\mu t_\mu f_{\alpha_\mu}
$$
is called the {\it normal remainder} of the
representation~$\mathcal{P}=(p_1f_1,\dots,p_mf_m)$
of~$f$ with respect to~$\mathcal{G}$.
\end{definition}

Clearly, the normal remainder depends on the choice of the
representation~$\mathcal{P}$. It has the following
application.

\begin{corollary}\label{sbpgenerates}
In the setting of the algorithm, the residue classes of the elements
of~$\mathcal{O}_F$ generate the image of the ideal~$J$ in
$P/\langle G\rangle$ as a $K$-vector space.

In other words, the residue class of every polynomial $f\in J$
can be represented as a $K$-linear combination of the
residue classes $\{\bar t_1 \bar f_{\alpha_1},\dots,
\bar t_\mu \bar f_{\alpha_\mu} \}$. Indeed, such a representation
can be found by computing the normal
remainder $\NR_{\mathcal{O}_F,\mathcal{G}}(\mathcal{P})$
for $\mathcal{G}=(g_1,\dots,g_\nu)$ and the representation
$\mathcal{P}=(p_1 f_1,\dots,p_m f_m)$ of
$f=p_1 f_1+\cdots+p_m f_m$.
\end{corollary}

\begin{proof}
By the algorithm, every $f\in J$ can be represented
in the form $f=h_1 g_1+\dots+h_\nu g_\nu+c_1 t_1 f_{\alpha_1}
+\dots + c_\mu t_\mu f_{\alpha_\mu}$,
where $h_1,\dots,h_\nu\in P$ and $c_1,\dots,c_\mu\in K$.
Forming residue classes modulo $\langle G\rangle$
yields the claim.
\end{proof}

Our last corollary provides another motivation for the name
``subideal border basis''.

\begin{corollary}\label{sbbgenerate}
In the setting of the algorithm, let~$G$
be an $\mathcal{O}_F$-subideal border basis of an ideal $I\subset P$.
Then~$G$ generates the ideal $I\cap J$.
\end{corollary}

\begin{proof} By definition, we have  $\langle g_1,\dots,g_\nu\rangle
\subseteq I\cap J$.
To prove the converse inclusion, let $f\in I\cap J$. Using the Subideal
Border Division Algorithm, the polynomial~$f$ can be expanded as
$f=h_1 g_1+\dots+h_\nu g_\nu+c_1 t_1 f_{\alpha_1} +\dots +
c_\mu t_\mu f_{\alpha_\mu}$,
where $h_1,\dots,h_\nu\in P$ and $c_1,\dots,c_\mu\in K$.
This implies the equality of residue classes
$0=\bar f= c_1 \bar t_1 \bar f_{\alpha_1} +\dots +
c_\mu \bar t_\mu \bar f_{\alpha_\mu}$ in $P/I$.
By assumption, the residue classes $\bar t_1 \bar f_{\alpha_1},
\dots,\bar t_\mu \bar f_{\alpha_\mu}$
form a $K$-vector space basis of $(I+J)/I$.
Hence $c_1=\dots=c_\mu=0$, and the expansion of~$f$ yields
$f=h_1 g_1+\dots+h_\nu g_\nu \in \langle G\rangle$.
\end{proof}

\bigskip

\section{The Subideal Version of the BM-Algorithm}

Let $K$ be a field, let $P=K[x_1,\dots,x_n]$ be the polynomial ring
in~$n$ indeterminates over~$K$, equipped with the standard grading, and let
$\mathbb{T}^n$ be the monoid of terms in~$P$. Given a finite set of points
$\mathbb{X}=\{p_1,\dots,p_s\} \subseteq K^n$, we let $\eval:
P \longrightarrow K^s$ be the evaluation map
$\eval(f)=(f(p_1),\dots,f(p_s))$ associated to~$\mathbb{X}$.
It is easy to adjust the Buchberger-M\"oller Algorithm (BM-Algorithm)
so that it computes a border basis of the {\it vanishing ideal}
$$
I_{\mathbb{X}}= \langle f\in P \mid f(p_1)= \cdots = f(p_s)=0\rangle
=\ker(\eval) \subseteq P
$$
of~$\mathbb{X}$. Since we use a version which differs slightly from the standard
formulation (see for instance~\cite{BM} or~\cite{KR2}, Thm.~6.3.10),
let us briefly recall its main steps.

\begin{algorithm}{\bf (BM-Algorithm for Border Bases)}\\
Let $\mathbb{X}=\{p_1,\dots,p_s\}\subseteq K^n$ be a set of
points given by their coordinates, and let~$\sigma$
be a degree compatible term ordering on~$\mathbb{T}^n$.
The following instructions define an algorithm which computes
the order ideal $\mathcal{O}_\sigma(I)=\mathbb{T}^n\setminus
\LT_\sigma(I_{\mathbb{X}})$ and the $\mathcal{O}_\sigma(I_{\mathbb{X}})$-border
basis~$G$ of~$I_{\mathbb{X}}$.

\begin{enumerate}
\item[{\bf B1}] Let $d=0$, $\mathcal{O}=\{1\}$, $G=\emptyset$,
and $\mathcal{M}=(1,\dots,1)^{\rm tr}\in\Mat_{s,1}(K)$.

\item[{\bf B2}] Increase~$d$ by one and let~$L=[t_1,\dots,t_\ell]$
be the list of all terms of degree~$d$ in~$\partial\mathcal{O}$,
ordered decreasingly w.r.t.~$\sigma$.
If $L=\emptyset$, return $(\mathcal{O},G)$ and stop.

\item[{\bf B3}] Form the matrix $\mathcal{A}=(\eval(t_1)\mid \cdots \mid \eval(t_\ell)
\mid \mathcal{M})$ and compute a matrix~$\mathcal{B}$ whose rows are a basis of the
kernel of~$\mathcal{A}$.

\item[{\bf B4}] Reduce~$\mathcal{B}$ to a matrix
$\mathcal{C}=(c_{ij}) \in\Mat_{k,\ell+m}(K)$ in row echelon form.

\item[{\bf B5}] For all $j\in\{1,\dots,\ell\}$ such that there exists
an $i\in\{1,\dots,k\}$ with pivot index $\nu(i)=j$, append the polynomial
$$
t_j + \sum_{j'=j+1}^\ell c_{ij'} t_{j'} + \sum_{j'=\ell+1}^{\ell+m} c_{ij'}u_{j'}
$$
to the list~$G$, where $u_{j'}$ is the $(j'-\ell)^{\rm th}$ element of~$\mathcal{O}$.

\item[{\bf B6}] For all $j=\ell,\ell-1,\dots,1$ such that the $j^{\rm th}$ column
of~$\mathcal{C}$ contains no pivot element, append the term~$t_j$
as a new first element to~$\mathcal{O}$, append the column $\eval(t_j)$ as a new
first column to~$\mathcal{M}$, and continue with step~{\bf B2}.

\end{enumerate}
\end{algorithm}

The proof of this modified version is simply obtained by
combining all the iterations of the usual BM-Algorithm corresponding to
terms of degree~$d$ into one ``block''. The fact that we put the
terms of degree~$d$ in~$\partial\mathcal{O}$ into~$L$
in step~{\bf B2} effects the computation of the entire border basis,
rather than just the reduced $\sigma$-Gr\"obner basis of~$\IX$
(see~\cite{HKPP}, Thm.\ 3.3). A further elaboration is beyond the scope of
the present paper and is left to the interested reader.

Given~$\mathbb{X}$ and a polynomial ideal $J=\langle F\rangle$
with $F=\{f_1,\dots,f_m\}\subset P\setminus \{0\}$, we know that
the vanishing ideal~$\IX$ has an $\mathcal{O}_\sigma(I)_F$-subideal
border basis. The following generalization of the
BM-algorithm computes this subideal border basis.

\begin{algorithm}{\bf (Subideal Version of the BM-Algorithm)}\label{SBMalg}\\
Let $\mathbb{X}=\{p_1,\dots,p_s\}\subseteq K^n $ be a set of
points given by their coordinates, let~$\sigma$ be a degree
compatible term ordering, and let $F=\{f_1,\dots,f_m\}\subset P\setminus \{0\}$
be a set of polynomials which generate an ideal $J=\langle F\rangle$.
The following instructions define an algorithm which computes an
$F$-order ideal $\mathcal{O}_\sigma(I)_F$ and the
$\mathcal{O}_\sigma(I)_F$-subideal border basis~$G$ of~$\IX$.

\begin{enumerate}
\item[{\bf S1}] Let $d=\min\{\deg(f_1),\dots,\deg(f_m)\}-1$,
$\mathcal{O}_F=\emptyset$, $G=\emptyset$,
and $\mathcal{M}\in\Mat_{s,0}(K)$.

\item[{\bf S2}] Increase~$d$ by one. Let $L=[t_1 f_{\alpha_1},\dots,t_\ell
f_{\alpha_\ell}]$ be the list of all $F$-terms of degree~$d$
in~$F\cup\partial\mathcal{O}_F$, with their leading terms
ordered decreasingly w.r.t.~$\sigma$.
If then $L=\emptyset$ and $d\ge \max\{\deg(f_1),\dots,\deg(f_m)\}$,
return $(\mathcal{O}_F,G)$ and stop.

\item[{\bf S3}] Form the matrix $\mathcal{A}=(\eval(t_1 f_{\alpha_1})\mid \cdots
\mid \eval(t_\ell f_{\alpha_\ell}) \mid \mathcal{M})$ and compute a matrix~$\mathcal{B}$
whose rows are a basis of the kernel of~$\mathcal{A}$.

\item[{\bf S4}] Reduce~$\mathcal{B}$ to a matrix
$\mathcal{C}=(c_{ij}) \in\Mat_{k,\ell+m}(K)$ in reduced row echelon form.

\item[{\bf S5}] For all $j\in\{1,\dots,\ell\}$ such that there exists
an $i\in\{1,\dots,k\}$ with pivot index $\nu(i)=j$, append the polynomial
$$
t_j f_{\alpha_j}+ \sum_{j'=j+1}^\ell c_{ij'} t_{j'}f_{\alpha_{j'}} +
\sum_{j'=\ell+1}^{\ell+m} c_{ij'}u_{j'}
$$
to the list~$G$, where $u_{j'}$ is the $(j'-\ell)^{\rm th}$ element
of~$\mathcal{O}_F$.

\item[{\bf S6}] For all $j=\ell,\ell-1,\dots,1$ such that the $j^{\rm th}$ column
of~$\mathcal{C}$ contains no pivot element, append the $F$-term~$t_j f_{\alpha_j}$
as a new first element to~$\mathcal{O}_F$, append the column $\eval(t_j f_{\alpha_j})$
as a new first column to~$\mathcal{M}$, and continue with step~{\bf S2}.

\end{enumerate}
\end{algorithm}

\begin{proof}
First we show finiteness. When a new degree is started in step~{\bf S2},
the matrix~$\mathcal{M}$ has $m=\#\mathcal{O}_F$ columns where~$\mathcal{O}_F$
is the {\it current} list of $F$-terms. In step~{\bf S6} we
enlarge~$\mathcal{M}$ by new first columns which are linearly independent of
the other columns. This can happen only finitely many times. Eventually
we arrive at a situation where all new columns $\eval(t_i f_{\alpha_i})$
of~$\mathcal{A}$ in step~{\bf S3} are linearly dependent on the previous
columns, and therefore the corresponding column of~$\mathcal{C}$
contains a pivot element.
Consequently, no elements are appended to~$\mathcal{O}_F$ in that
degree and we get $L=\emptyset$ in the next degree. Hence the algorithm stops.

Now we show correctness. The columns
of~$\mathcal{A}$ are the evaluation vectors of $F$-terms whose
leading terms are ordered decreasingly w.r.t.~$\sigma$.
A row $(c_{i1},\dots,c_{i\,\ell{+}m})$ of~$\mathcal{C}$ corresponds
to a linear combination of these $F$-terms whose evaluation vector
is zero. Let $g_1,\dots,g_k$ be the polynomials given
by these linear combinations of $F$-terms. Clearly, we have $g_i\in
\IX \cap J$.

The evaluation vectors of the $F$-terms which are put into~$\mathcal{O}_F$
in step~{\bf S6} are linearly independent of the evaluation vectors of the $F$-terms
in the previous set~$\mathcal{O}_F$ since there is no linear relation
leading to a pivot element in the corresponding column of~$\mathcal{C}$.
Inductively it follows that the evaluation vectors of the $F$-terms
in~$\mathcal{O}_F$ are always linearly independent. Henceforth
the pivot elements of~$\mathcal{C}$ are always in the ``new'' columns
and the polynomials~$g_i$ have degree~$d$.
By the way the algorithm proceeds, every $F$-term in the border of the final
set~$\mathcal{O}_F$ appears in exactly one on the elements of~$G$.
All the other summands of a polynomial $g_i$ are in~$\mathcal{O}_F$.
Hence the final set~$G$ is an $\mathcal{O}_F$-subideal border prebasis.

Furthermore, every $F$-term is either in~$\mathcal{O}_F$ or it is
a multiple of an $F$-term in $\partial\mathcal{O}_F$
(cf.\ Prop.~\ref{IndexProps}.5). In the latter case, its
evaluation vector can be written as a linear combination of the
evaluation vectors of the elements of~$\mathcal{O}_F$.
Thus the evaluation vectors of the elements of~$\mathcal{O}_F$
generate the space of all evaluation vectors of $F$-terms.
Since they are linearly independent, they form a $K$-basis of that space.
Now we use the facts that evaluation yields an isomorphism of
$K$-vector spaces $\overline{{\rm eval}}: P/I \longrightarrow K^s$
and that the residue classes of the $F$-terms generate the
$K$-vector subspace $(I+J)/I$ of~$P/I$ to conclude that
the residue classes of the $F$-terms in the final set~$\mathcal{O}_F$
form a $K$-basis of $(I+J)/I$.
\end{proof}

Let us illustrate this algorithm by an example.

\begin{example}\label{SBMexample}
In the polynomial ring $P=\mathbb{Q}[x,y,z]$, we consider the ideal
$J=\langle F\rangle$ with $F=\{f_1,\,f_2\}$ given by $f_1=x^2-1$ and $f_2=y-z$.
Let $\sigma={\tt DegRevLex}$.

We want to compute an $\mathcal{O}_F$-subideal border basis
of the vanishing ideal of the point set
$\mathbb{X}=\{(1,1,1),\, (0,1,1),\, (1,1,0),\, (1,0,1)\}$.
Notice that the first point of~$\mathbb{X}$ lies on
$\mathcal{Z}(f_1,f_2)$, so that we should expect an $F$-order
ideal consisting of three $F$-terms. Let us follow the steps of
the algorithm. (We only list those steps in which something happens.)

\begin{enumerate}
\item[{\bf S2}] Let $d=1$ and $L=[y-z]$.

\item[{\bf S3}] Form $\mathcal{A}=(0,0,1,-1)^{\rm tr}$
and compute $\mathcal{B}=(0)$. (Thus $\mathcal{C}=\mathcal{B}$.)

\item[{\bf S6}] Let $\mathcal{O}_F=\{y-z\}$ and $\mathcal{M}=(0,0,1,-1)^{\rm tr}$.

\item[{\bf S2}] Let $d=2$ and $L=[ x^2-1,\, x(y-z),\, y(y-z),\, z(y-z)]$.

\item[{\bf S3}] Compute $\mathcal{A}=\begin{pmatrix}
0 & 0 & 0 & 0 & 0 \\ -1 & 0 & 0 & 0 & 0 \\ 0 & 1 & 1 & 0 & 1 \\ 0 & -1 &
0 & -1 & -1
\end{pmatrix}$ and
$\mathcal{B}= \begin{pmatrix} 0 & 1 & 0 & 0 & -1 \\ 0 & 0 & 1 & 1 & -1
\end{pmatrix}$. (Thus $\mathcal{C}=\mathcal{B}$.)

\item[{\bf S5}] The pivot indices $\nu(1)=2$ and $\nu(2)=3$ yield the set
$G=\{g_1,g_2\}$ with $g_1=x(y-z)-(y-z)$ and $g_2=y(y-z)+z(y-z)-(y-z)$.

\item[{\bf S6}] We obtain $\mathcal{O}_F=\{x^2-1,\, z(y-z),\, y-z\}$ and
$\mathcal{M}=\begin{pmatrix}
0 & 0 & 0 \\
\!\!\! -1 & 0 & 0 \\
0 & 0 & 1\\
0 & \!\!\! -1 & \!\!\! -1
\end{pmatrix}$.

\item[{\bf S2}] Let $d=3$. We have $L=[x(x^2-1),\,y(x^2-1),\, z(x^2-1),\,
xz(y-z),\,yz(y-z),\, z^2(y-z)]$.

\allowbreak
\item[{\bf S3}] Find $\mathcal{A}=\begin{pmatrix}
0 & 0 & 0 & 0         & 0 & 0            & 0 & 0 & 0 \\
0 & 0 & 0 & 0         & 0 & 0            & \!\!\! -1 & 0 & 0 \\
0 & 0 & 0 & 0         & 0 & 0            & 0 & 0 & 1 \\
0 & 0 & 0 & \!\!\! -1 & 0 & \!\!\! -1    & 0 & \!\!\! -1 & \!\!\! -1
\end{pmatrix}$
\vskip 3pt\allowbreak
and $\mathcal{B}=\begin{pmatrix}
1 & 0 & 0 & 0 & 0 & 0 & 0 & 0 & 0\\
0 & 1 & 0 & 0 & 0 & 0 & 0 & 0 & 0\\
0 & 0 & 1 & 0 & 0 & 0 & 0 & 0 & 0\\
0 & 0 & 0 & 1 & 0 & 0 & 0 & \!\!\! -1 & 0\\
0 & 0 & 0 & 0 & 1 & 0 & 0 & 0 & 0\\
0 & 0 & 0 & 0 & 0 & 1 & 0 & \!\!\! -1 & 0
\end{pmatrix}$. (Thus $\mathcal{C}=\mathcal{B}$.)
\vskip 3pt

\item[{\bf S5}] Here we obtain $G=\{g_1,\dots,g_8\}$ where
$g_3=x(x^2-1)$, $g_4=y(x^2-1)$, $g_5=z(x^2-1)$,
$g_6=xz(y-z)-z(y-z)$, $g_7=yz(y-z)$, and finally
$g_8=z^2(y-z)-z(y-z)$.

\item[{\bf S6}] There are no new non-pivot indices. Hence
$\mathcal{O}$ and $\mathcal{M}$ are not changed.

\item[{\bf S2}] We get $L=\emptyset$
and the algorithm stops.
\end{enumerate}

The result is the $F$-order ideal
$\mathcal{O}_F=\{x^2-1,\, z(y-z),\, y-z\}$
and the $\mathcal{O}_F$-subideal border basis
$G=\{g_1,\dots,g_8\}$ of~$I_{\mathbb{X}}$.
\end{example}

\bigskip

\section{The Subideal Version of the AVI-Algorithm}

From here on we work in the polynomial ring $P=\mathbb{R}[x_1,\dots,x_n]$
over the field of real numbers. We let $\mathbb{X}=\{p_1,\dots,p_s\}\subset [-1,1]^n
\subset \mathbb{R}^n$ be a finite set of points and $\epsilon> \tau >0$
two threshold numbers. (The number~$\epsilon$ can be thought of as a
measure for error tolerance of the input data points~$\mathbb{X}$
and~$\tau$ is used as a ``minimum size'' for acceptable leading coefficients
of unitary polynomials.)

Let us point out the following notational convention
we are using: the ``usual'' norm of a polynomial $f\in P$
is the Euclidean norm of its coefficient vector and is
denoted by~$\|f\|$. By ``unitary'' we mean $\|f\|=1$.
In contrast, by $\|f\|_1$ we mean the sum of the absolute
values of the coefficients of~$f$, and the term
``$\|\;\|_1$-unitary'' is to be interpreted accordingly.

Furthermore, by $\eval: P \longrightarrow \mathbb{R}^s$ we denote the evaluation map
$\eval(f)=(f(p_1),\dots,f(p_s))$ associated to~$\mathbb{X}$.
For the convenience of the reader, we briefly recall the basic structure
of the Approximate Vanishing Ideal Algorithm (AVI-algorithm) from~\cite{HKPP}.
Notice that we skip several technical details and explicit error estimates.
The goal of the AVI-algorithm is to compute an approximate border
basis, a notion that is defined as follows.

\begin{definition}
Let $\mathcal{O}=\{t_1,\dots,t_\mu\}\subseteq \mathbb{T}^n$ be an order
ideal of terms, let $\partial\mathcal{O}=\{b_1,\dots,b_\nu\}$ be its
border, and let $G=\{g_1,\dots,g_\nu\}$ be an $\mathcal{O}$-border
prebasis of the ideal $I=\langle g_1,\dots,g_\nu\rangle$ in~$P$.
Recall that this means that $g_j$ is of the form $g_j=b_j -
\sum_{i=1}^\mu c_{ij}t_i$ with $c_{ij}\in \mathbb{R}$.

For every pair $(i,j)$ such that $b_i,b_j$ are neighbors in $\partial\mathcal{O}$,
we compute the normal remainder $S'_{ij}= \NR_{\mathcal{O},G}(S_{ij})$
of the S-polynomial of~$g_i$ and~$g_j$ with respect to~$G$.
We say that~$G$ is an {\it $\epsilon$-approximate border basis}\/
of the ideal $I=\langle G\rangle$ if we have $\| S_{ij} \|< \epsilon$
for all such pairs $(i,j)$.
\end{definition}

Moreover, the AVI-algorithm uses the concepts of
approximate vanishing, approximate kernel and stabilized reduced
row echelon form, for which we refer to~\cite{HKPP}, Sect.~2 and~3.

\begin{algorithm}{\bf (AVI-Algorithm)}\label{AVI}\quad\\
Let $\mathbb{X}=\{p_1,\dots,p_s\}\subset [-1,1]^n \subset\mathbb{R}^n$
be a set of points as above, and let~$\sigma$ be a degree compatible
term ordering. Consider the following sequence of instructions.

\begin{enumerate}
\item[{\bf A1}] Start with lists $G=\emptyset$, $\mathcal{O}=[1]$,
a matrix $\mathcal{M}=(1,\dots,1)^{\rm tr}\in \Mat_{s,1}(\mathbb{R})$,
and $d=0$.

\item[{\bf A2}] Increase~$d$ by one and let~$L=[t_1,\dots,t_\ell]$
be the list of all terms of degree~$d$ in~$\partial\mathcal{O}$,
ordered decreasingly w.r.t.~$\sigma$.
If $L=\emptyset$, return the pair $(\mathcal{O},G)$ and stop.

\item[{\bf A3}] Form the matrix
$\mathcal{A}=(\eval(t_1),\dots,\eval(t_\ell),\mathcal{M})$
and calculate a matrix~$\mathcal{B}$
whose rows are an ONB of the approximate kernel $\apker(\mathcal{A},\epsilon)$
of~$\mathcal{A}$.

\item[{\bf A4}] Compute the stabilized reduced row echelon
form of~$\mathcal{B}$ with respect to the given~$\tau$.
The result is a matrix $\mathcal{C}=(c_{ij}) \in\Mat_{k,\ell+m}(\mathbb{R})$
such that $c_{ij}=0$ for $j<\nu(i)$. Here $\nu(i)$ denotes the
column index of the pivot element in the $i^{\rm th}$ row of~$\mathcal{C}$.

\item[{\bf A5}] For all $j\in\{1,\dots,\ell\}$ such that there exists
an $i\in\{1,\dots,k\}$ with $\nu(i)=j$, append the polynomial
$$
c_{ij}t_j + \sum_{j'=j+1}^\ell c_{ij'} t_{j'} + \sum_{j'=\ell+1}^{\ell+m} c_{ij'}u_{j'}
$$
to the list~$G$, where $u_{j'}$ is the $(j'-\ell)^{\rm th}$ element of~$\mathcal{O}$.

\item[{\bf A6}] For all $j=\ell,\ell-1,\dots,1$ such that the $j^{\rm th}$ column
of~$\mathcal{C}$ contains no pivot element, append the term~$t_j$
as a new first element to~$\mathcal{O}$ and append the column $\eval(t_j)$ as a new
first column to~$\mathcal{M}$.

\item[{\bf A7}] Calculate a matrix~$\mathcal{B}$ whose
rows are an ONB of $\apker(\mathcal{M},\epsilon)$.

\item[{\bf A8}] Repeat steps {\bf A4} -- {\bf A7}
until $\mathcal{B}$ is empty. Then continue with step~{\bf A2}.

\end{enumerate}

\noindent This is an algorithm which computes a pair $(\mathcal{O},G)$
such that the following properties hold for the bounds~$\delta$ and~$\eta$
given in~\cite{HKPP}, Thm.\ 3.3.

\begin{enumerate}
\item[(a)] The set~$G$ consists of unitary polynomials
which vanish $\delta$-approximately at the points of~$\mathbb{X}$.

\item[(b)] The set~$\mathcal{O}=\{t_1,\dots,t_\mu\}$
contains an order ideal of terms
such that there is no unitary polynomial in $\langle \mathcal{O}\rangle_K$
which vanishes $\epsilon$-approximately on~$\mathbb{X}$.

\item[(c)] The set $\widetilde{G}=\{(1/\LC_\sigma(g))\,g \mid g\in G\}$ is an
$\mathcal{O}$-border prebasis.

\item[(d)] The set~$\widetilde{G}$ is an $\eta$-approximate border basis.

\end{enumerate}
\end{algorithm}

Our main algorithm combines the techniques of this AVI-algorithm
with the subideal version of the BM-algorithm presented above
(see Alg.~\ref{SBMalg}). The result is an algorithm which
computes an approximate subideal border basis. This notion is defined
as follows.

\begin{definition}
Let $\mathcal{O}_F=\{t_1f_{\alpha_1},\dots,t_\mu f_{\alpha_\mu}\}$
be an $F$-order ideal, let $\partial\mathcal{O}_F=\{b_1 f_{\beta_1},
\dots,b_\nu f_{\beta_\nu}\}$ be its
border, and let $G=\{g_1,\dots,g_\nu\}$ be an $\mathcal{O}_F$-subideal
border prebasis.
Recall that this means that $g_j$ is of the form $g_j=b_j f_{\beta_j}-
\sum_{i=1}^\mu c_{ij}t_i f_{\alpha_i}$ with $c_{ij}\in \mathbb{R}$.

For every pair $(i,j)$ such that $b_i,b_j$ are neighbors in
$\partial\mathcal{O}_F$, i.e.\ such that $\beta_i=\beta_j$ and
$b_i,b_j$ are neighbors in the usual sense,
we compute the normal remainder $S'_{ij}= \NR_{\mathcal{O}_F,G}(S_{ij})$
of the S-polynomial of~$g_i$ and~$g_j$ with respect to~$G$.
We say that~$G$ is an {\it $\epsilon$-approximate $\mathcal{O}_F$-subideal
border basis}\/ if we have $\| S_{ij} \|< \epsilon$
for all such pairs $(i,j)$.
\end{definition}

Now we are ready to formulate and proof the main result of this section.

\begin{algorithm}{\bf (Subideal Version of the AVI-Algorithm)}\label{SAVI}
\quad\\
Let $\mathbb{X}=\{p_1,\dots,p_s\}\subset [-1,1]^n \subset\mathbb{R}^n$
be a set of points as above, let~$\sigma$ be a degree compatible
term ordering, and let $F=\{f_1,\dots,f_m\}\subset P\setminus \{0\}$
be a set of $\|\;\|_1$-unitary polynomials which generate an
ideal $J=\langle F\rangle$. Consider the following sequence of instructions.

\begin{enumerate}
\item[{\bf SA1}] Let $d=\min\{\deg(f_1),\dots,\deg(f_m)\}-1$,
$\mathcal{O}_F=\emptyset$, $G=\emptyset$,
and $\mathcal{M}\in\Mat_{s,0}(K)$.

\item[{\bf SA2}] Increase~$d$ by one. Let $L=[t_1 f_{\alpha_1},\dots,t_\ell
f_{\alpha_\ell}]$ be the list of all $F$-terms of degree~$d$
in~$F\cup\partial\mathcal{O}_F$, with their leading terms
ordered decreasingly w.r.t.~$\sigma$.
If then $L=\emptyset$ and $d\ge \max\{\deg(f_1),\dots,\deg(f_m)\}$,
return $(\mathcal{O}_F,G)$ and stop.

\item[{\bf SA3}] Form the matrix $\mathcal{A}=(\eval(t_1 f_{\alpha_1})\mid \cdots
\mid \eval(t_\ell f_{\alpha_\ell}) \mid \mathcal{M})$ and compute a matrix~$\mathcal{B}$
whose rows are an ONB of the approximate kernel of~$\mathcal{A}$.

\item[{\bf SA4}] Compute the stabilized reduced row echelon
form of~$\mathcal{B}$ with respect to the given~$\tau$.
The result is a matrix $\mathcal{C}=(c_{ij}) \in\Mat_{k,\ell+m}(\mathbb{R})$
such that $c_{ij}=0$ for $j<\nu(i)$. Here $\nu(i)$ denotes the
column index of the pivot element in the $i^{\rm th}$ row of~$\mathcal{C}$.

\item[{\bf SA5}] For all $j\in\{1,\dots,\ell\}$ such that there exists
an $i\in\{1,\dots,k\}$ with $\nu(i)=j$, append the polynomial
$$
t_j f_{\alpha_j}+ \sum_{j'=j+1}^\ell c_{ij'} t_{j'}f_{\alpha_{j'}} +
\sum_{j'=\ell+1}^{\ell+m} c_{ij'}u_{j'}
$$
to the list~$G$, where $u_{j'}$ is the $(j'-\ell)^{\rm th}$ element
of~$\mathcal{O}_F$.

\item[{\bf SA6}] For all $j=\ell,\ell-1,\dots,1$ such that the $j^{\rm th}$ column
of~$\mathcal{C}$ contains no pivot element, append the $F$-term~$t_j f_{\alpha_j}$
as a new first element to~$\mathcal{O}_F$, append the column $\eval(t_j f_{\alpha_j})$
as a new first column to~$\mathcal{M}$.

\item[{\bf SA7}] Calculate a matrix~$\mathcal{B}$ whose
rows are an ONB of $\apker(\mathcal{M},\epsilon)$.

\item[{\bf SA8}] Repeat steps {\bf SA4} -- {\bf SA7}
until $\mathcal{B}$ is empty. Then continue with step~{\bf A2}.

\end{enumerate}

\noindent This is an algorithm which computes a pair $(\mathcal{O}_F,G)$
with the following properties:

\begin{enumerate}
\item[(a)] The set~$G$ consists of unitary polynomials
which vanish $\delta$-approximately at the points of~$\mathbb{X}$.
Here we can use $\delta= \epsilon\sqrt{\nu}+\tau\nu(\mu+\nu)\sqrt{s}$.

\item[(b)] The set~$\mathcal{O}_F$ contains an $F$-order ideal
such that there is no unitary polynomial in $\langle \mathcal{O}_F\rangle_K$
which vanishes $\epsilon$-approximately on~$\mathbb{X}$.

\item[(c)] The set $\widetilde{G}=\{(1/\LC_\sigma(g))\,g \mid g\in G\}$ is an
$\mathcal{O}_F$-subideal border prebasis.

\item[(d)] The set~$\widetilde{G}$ is an $\eta$-approximate subideal
border basis for
$\eta=2\delta+2\nu\delta^2/\gamma\epsilon + 2\nu\delta\sqrt{s}/\epsilon$.
Here~$\gamma$ denotes the smallest absolute value of the border
$F$-term coefficient of one the polynomials~$g_i$.
\end{enumerate}
\end{algorithm}

\begin{proof}
Large parts of this proof correspond exactly to the
proof of the usual AVI-algorithm (see Thm.~3.2 in~\cite{HKPP}).
Therefore we will mainly point of the additional arguments
necessary to show the subideal version.
The finiteness proof is identical to the finiteness proof in the
subideal version of the BM-algorithm~\ref{SBMalg}.

For the proof of~(a), we can proceed exactly as in the
case of the usual AVI-algorithm. There is only one point where we have
to provide a further argument: the norm of the evaluation vector
of an $F$-term is~$\le \sqrt{s}$. To see this, we let $t_i f_j$ be an
$F$-term and we write $t_i f_j=\sum_k c_k \tilde t_k$ with $c_k \in
\mathbb{R}$ and $\tilde t_k \in\mathbb{T}^n$. Since~$f_j$ is $\|\;\|_1$-unitary
and $\mathbb{X}\in [-1,1]^n$, we have $\| \eval(t_i f_j) \| \le
\sum_k |c_k|\, \|\eval(\tilde t_k)\| \le \| f_j \|_1\,\sqrt{s}= \sqrt{s}$.

Next we show~(b). The columns of the final matrix~$\mathcal{M}$
are precisely the evaluation vectors of the $F$-terms in~$\mathcal{O}_F$.
After the loop in steps {\bf SA4} -- {\bf SA8}, we have $\apker(\mathcal{M})=\{0\}$.
Hence no unitary polynomial in $\langle \mathcal{O}_F\rangle_K$ has an
evaluation vector which is smaller than~$\epsilon$.
It remains to show that~$\mathcal{O}_F$ is an $F$-order ideal.
Suppose that $t_i f_j\in\mathcal{O}_F$ and that $x_k t_i f_j$ is put
into~$\mathcal{O}_F$. We have to prove that every $F$-term
$\tilde t \, f_j$ such that $x_\ell \tilde t f_j= x_k t_i f_j$
is also contained in~$\mathcal{O}_F$. In this case we have $t_i=x_\ell t'$
and we want to show $x_k t'f_j\in\mathcal{O}_F$. For a contradiction,
suppose that $x_k t'f_j$ is the border $F$-term of some $g\in G$.
Since the evaluation vector of~$x_\ell x_k t' f_j=x_k t_i f_j$
is not larger than $\eval(x_k t' f_j)$, also this $F$-term would be
detected by the loop of steps {\bf SA4} -- {\bf SA8} as
the border $F$-term of an element of~$G$. This contradicts
$x_k t_i f_j\in \mathcal{O}_F$.

To prove~(c), it suffices to note that steps~{\bf SA2} and~{\bf SA5}
make sure that the elements of~$G$ have the necessary form. Finally,
claim~(d) follows in exactly the same way as part~(d) of~\cite{HKPP},
Thm.~3.3.
\end{proof}

Let us follow the steps of this algorithm in a concrete case which is a slightly
perturbed version of Example~\ref{SBMexample}.

\begin{example}
In the ring $P=\mathbb{R}[x,y,z]$ we consider the ideal
$J=\langle f_1,f_2\rangle $ generated by the $\|\;\|_1$-unitary
polynomials $f_1=0.5\,y -0.5\, z$ and $f_2=0.5\, x^2-0.5$.
Let $\sigma={\tt DegRevLex}$, let $\epsilon=0.03$, and let
$\tau=0.001$. We want to compute an approximate
subideal border basis vanishing approximately at
the points of
$\mathbb{X}=\{(1,1,1),\, (0,1,1),\, (1,1,0),\, (1,0,0.98),\,
(0.98,0,1)$.

Notice that the first point of~$\mathbb{X}$ is contained in $\mathcal{Z}(f_1,f_2)$
and that the last two points of~$\mathbb{X}$ differ by $\le \epsilon$ from
one point $(1,0,1)$. Hence the approximate subideal border basis should correspond
to {\it three}\/ points outside $\mathcal{Z}(J)$, and therefore we should
expect to get an $F$-order ideal consisting of three $F$-terms.
We follow the steps of the subideal version of the AVI-algorithm~\ref{SAVI}.

\begin{enumerate}
\item[{\bf SA2}] Let $d=1$ and $L=[0.5\,y - 0.5\,z]$.

\item[{\bf SA3}] We compute $\mathcal{A}=(0,0,0.5,-0.49,-0.51)^{\rm tr}$ and
$\mathcal{B}=(0)$. (Thus $\mathcal{C}=\mathcal{B}$.)

\item[{\bf SA6}] Let $\mathcal{O}=\{f_1\}$ and $\mathcal{M}=
(0,0,0.5,-0.49,-0.5)^{\rm tr}$.

\item[{\bf SA2}] Let $d=2$ and $L=[f_2,\, xf_1,\, yf_1,\, zf_1]$.
\vskip 3pt

\item[{\bf SA3}] We compute $\mathcal{A}=\begin{pmatrix}
0 & 0 & 0 & 0 & 0 \\ -0.5 & 0 & 0 & 0 & 0 \\
0 & 0.5 & 0.5 & 0 & 0.5 \\ 0 & -0.49 & 0 & -0.4802 & -0.49 \\
-0.0198 & -0.49 & 0 & -0.5 & -0.5
\end{pmatrix}$ and
\vskip 3pt
$\mathcal{B}= \apker(\mathcal{A},\epsilon)=\begin{pmatrix}
0.0004 & 0.6755 & -0.5089 & -0.5068 & -0.1667 \\
0 & -0.3812 & -0.3735 & -0.3812 & 0.7548
\end{pmatrix}$.

\item[{\bf SA4}] The stabilized reduced row echelon form of~$\mathcal{B}$ is
\vskip 5pt
$\mathcal{C}= \begin{pmatrix}
0 & 0.3812 & 0.3735 & 0.3812 & -0.7548 \\
0 & 0 & 0.5754 & 0.5811 & -0.5754
\end{pmatrix}$.
\vskip 3pt

\item[{\bf SA5}] We get $G=\{g_1,g_2\}$ with $g_1= 0.3812\, xf_1 +
0.3735\, yf_1 + 0.3812\, zf_1 - 0.7548 f_1$
and $g_2= 0.5754\, yf_1 + 0.5811\, zf_1 - 0.5754 f_1$.

\item[{\bf SA6}] We find $\mathcal{O}=\{f_2,\, zf_1,\, f_1\}$
and $\mathcal{M}= \begin{pmatrix}
0 & -0.5 & 0 & 0 & 0 \\
0 & 0 & 0 & -0.4802 & -0.5\\
0 & 0 & 0.5 & -0.49 & -0.5
\end{pmatrix}^{\!\!\rm tr}$.

\item[{\bf SA2}] Now let $d=3$ and $L=[xf_2,\,yf_2,\,zf_2, xzf_1,\, yzf_1,\, z^2f_1]$.
\vskip 3pt

\item[{\bf SA3}] $\mathcal{A}= \begin{pmatrix}
0 & 0 & 0 & 0 & 0 & 0 & 0 & 0 & 0\\
0 & -0.5 & -0.5 & 0 & 0 & 0 & -0.5 & 0 & 0\\
0 & 0 & 0 & 0 & 0 & 0 & 0 & 0 & 0.5 \\
0 & 0 & 0 & -0.48 & 0 & -0.47 & 0 & -0.48 & -0.49\\
-0.02 & 0 & -0.02 & -0.49 & 0 & -0.5 & -0.02 & -0.5 & -0.5
\end{pmatrix}$
\vskip 3pt
\noindent and $\mathcal{B}, \mathcal{C}$ are matrices
of rank~6 which yield six further approximate subideal border basis
elements.

\item[{\bf SA5}] We obtain $G=\{g_1,\dots,g_8\}$ with
$g_3=xf_2-0.02\, zf_1$, $g_4=0.71\, yf_2 -0.71\, f_2 + 0.01\, zf_1$,
$g_5=0.71\, zf_2 -0.71\, f_2$, $g_6=0.71\, xzf_1- 0.7\,zf_1$,
$g_7=yzf_1$, and $g_8=0.71\, z^2f_1 -0.7\, zf_1$.

\item[{\bf SA5}] Since there is no new non-pivot row index, $\mathcal{O}_F$ and
$\mathcal{M}$ are not changed.

\item[{\bf SA2}] In degree $d=4$ we find $L=\emptyset$ and the algorithm stops.
\end{enumerate}

Hence the result is the $F$-order ideal $\mathcal{O}=\{x^2-1,\,z(y-z),\,y-z\}$
and the approximate $\mathcal{O}_F$-subideal border basis $G=\{g_1,\dots,g_8\}$.
This confirms that there are three approximate zeros of~$G$
outside the two lines $\mathcal{Z}(f_1,f_2)$.
\end{example}

\bigskip

\section{An Industrial Application}

In this section we apply the subideal version of the
AVI-algorithm to an actual industrial problem which has
been studied in the Algebraic Oil Research Project (see~\cite{AO}).
Viewed from a more general perspective, this
application shows how one can carry out the suggestion made in the
introduction, namely to use the subideal version of the AVI-algorithm
to introduce knowledge about the nature of a physical system
into the modeling process.

Suppose that a multi-zone well consists of two zones~$A$
and~$B$. During so-called {\it commingled production}, the two
zones are interacting and influence each other.
We have at our disposal time series of measured data such as pressures,
temperatures, total production and valve positions.
Moreover, during so-called {\it test phases} we can
obtain time series of these data when only one of the
two zones is producing. The following figure gives a
schematic representation of the physical system and the
measured variables.

\begin{figure}[ht]
\begin{centering}
\includegraphics[scale=0.5]{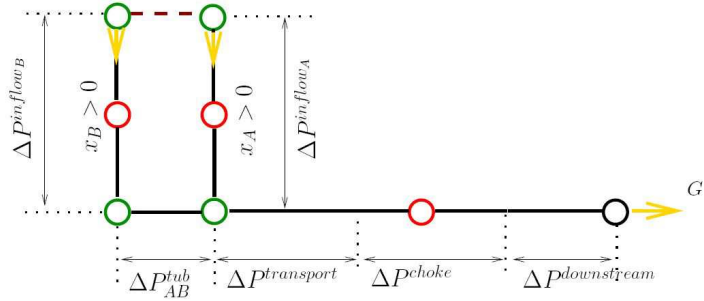}
\caption{Schematic representation of a two-zone well}
\end{centering}
\end{figure}

The measured total production does not equal the sum of the
individual productions calculated from the test data.
The {\it production allocation problem}\/ is to determine the contributions
of the two zones to the total production when they are producing
together. Here the {\it contributions} $c_A, c_B$ of the zones
are defined to be the part of the total production~$p_{AB}$ passing through
the corresponding down-hole valves. Therefore we have $p_{AB}=c_A+
c_B$, but there is no way of measuring~$c_A$ and~$c_B$ directly.
In this sense the production allocation problem is to determine
the contributions $c_A,c_B$ from the measured data.

Let the indeterminate~$x_A$ represent the valve position
of zone~$A$ and~$x_B$ the valve position of zone~$B$.
Here $x_i=0$ means that the valve is closed and
$x_i=1$ represents a fully opened valve position.
Clearly, if valve~$A$ is closed, i.e.\ for points in the
zero set $\mathcal{Z}(\langle x_A\rangle)$, there is no
contribution from zone~$A$, and likewise for~$B$.
By Hilbert's Nullstellensatz, this means that the polynomial
$p_A$ modeling the production of zone~$A$ should be computed
by using the subideal version of the AVI-algorithm with
$J=\langle x_A\rangle$. Similarly, we want to force
$p_B \in \langle x_B \rangle$.

Now we model the total production~$p_{AB}$ in the following
way. We write $p_{AB}=p_A+p_B+q_{AB}$ where~$q_{AB}$ is a polynomial
which measures the interaction of the two zones.
To compute~$q_{AB}$, we write it in the form
$$
q_{AB} = f_A \cdot (x_B\cdot p_A) + f_B \cdot (x_A\cdot p_B)
$$
Notice that such a decomposition can be computed via
the subideal version of the AVI-algorithm by applying it
to the ideal $J=\langle x_Bp_A,\, x_A p_B\rangle$.
The result will be a representation
$p_{AB}= p_A+ p_B + f_A x_B p_A + f_B x_A p_B$.
Here we observe that $x_A=0$ implies $p_{AB}= p_B$
because $p_A \in \langle x_A\rangle$. Analogously, we see that
$x_B=0$ implies $p_{AB}=p_A$, in accordance with the physical
situation.

The endresult of these computations is that the contributions
of the two zones during commingled production can be
computed from the equalities $c_A=(1+f_A x_B)p_A$ and
$c_B=(1+f_B x_A)p_B$. At the same time we gain a detailed
insight into the nature of the interactions by examining
the structure of the polynomials $f_A,f_B$.

\bigbreak

\subsection*{Acknowledgements}
The idea to construct a subideal version of the
AVI-algo\-rithm originated in discussions of the authors
with Daniel Heldt who also implemented a rough first
prototype. The algorithms of this paper have been implemented
by Jan Limbeck in the \apcocoa\ library (see~\cite{ApCoCoA})
and are freely available. The authors thank both of them for
the opportunity to use these implementations in the preparation
of this paper and in the Algebraic Oil Research Project
(see~\cite{AO}). Special thanks go to Lorenzo Robbiano
for useful discussions and to the Dipartimento die Matematica
of Universit\`a di Genova (Italy) for the hospitality the
authors enjoyed during part of the writing of this paper.


\end{document}